\documentclass[11pt]{amsart}

\usepackage{amsmath,amssymb,amscd,amsthm,latexsym,graphics,enumerate,engord,hyperref,txfonts}
\usepackage[all]{xy}
\usepackage{pstricks-add}
\usepackage[applemac]{inputenc}

\usepackage{xcolor}
\usepackage{soul}
\sethlcolor{yellow}
\usepackage{amsthm,amsmath}

\newcounter{first}
\newenvironment{probs}
{%
 \begin{list}
       {(\textrm{\thefirst})}
       {\usecounter{first} \setlength{\leftmargin}{5pt}
        \setlength{\topsep}{3pt}
        \setlength{\itemsep}{3pt}
       }
}%
{\end{list}}

\newtheorem{thm}{Theorem}
\newtheorem{lem}[thm]{Lemma}
\newtheorem{prop}[thm]{Proposition}
\newtheorem{cor}[thm]{Corollary}

\theoremstyle{definition}
\newtheorem{rem}[thm]{Remark}
\newtheorem{ex}[thm]{Example}

\def\R{{\mathbb{R}}}
\def\Z{{\mathbb{Z}}}
\def\P{{\mathbb{P}}}
\def\XX{{\mathcal{P}}}

\DeclareMathOperator{\TC}{{\sf TC}}
\DeclareMathOperator{\cat}{\mathrm{cat}}

\begin{document}
	
\title[On $\TC(M)$ when $\pi_1(M)$ is abelian]{On the topological complexity of manifolds with abelian fundamental group}

\author{Daniel C. Cohen}
\address{Department of Mathematics, Louisiana State University, Baton Rouge, LA 70803 USA}
\email{\href{mailto:cohen@math.lsu.edu}{cohen@math.lsu.edu}}
\urladdr{\href{http://www.math.lsu.edu/~cohen/}{www.math.lsu.edu/\char'176cohen}}

\author{Lucile Vandembroucq}
\address{Centro de Matem\'atica, Universidade do Minho, Campus de Gualtar, 4710-057 Braga, Portugal}
\email{\href{mailto:lucile@math.uminho.pt}{lucile@math.uminho.pt}}
\thanks{The second author is partially supported by Portuguese Funds through FCT -- Funda\c c\~ao para a Ci\^encia e a Tecnologia, within the projects UIDB/00013/2020 and UIDP/00013/2020. }

\begin{abstract} 
We find conditions which ensure that the topological complexity of a closed manifold $M$ with abelian fundamental group is nonmaximal, and see through examples that our conditions are sharp. This generalizes 
results of Costa and Farber on the topological complexity of spaces with small fundamental group. Relaxing the commutativity condition on the fundamental group, we also generalize results of Dranishnikov on the Lusternik-Schnirelmann category of the cofibre of the diagonal map $\Delta: M \to M \times M$ for nonorientable surfaces by establishing the nonmaximality of this invariant for a large class of manifolds. 	
\end{abstract}



\maketitle

\section{Introduction}
Let $X$ be a finite dimensional path-connected CW complex. 
Viewing $X$ as the configuration space of a mechanical system, Farber's topological complexity \cite{Far} is a numerical homotopy invariant which provides a measure of the complexity of the problem of motion planning in 
$X$, of interest in robotics. 
We consider the normalized versions of the topological complexity $\TC(X)$, and of the Lusternik-Schnirelmann (LS) category $\cat(X)$. That is, $\TC(X)$ is the least integer $k$ such that 
$X\times X$ can be covered by $k+1$ open sets 
on each of which 
the fibration $$X^{[0,1]}\to X\times X,\quad \lambda\mapsto (\lambda(0),\lambda(1))$$
admits a continuous (local) section, while $\cat(X)$ is the least integer $k$ such that 
$X$ can be covered by $k+1$ open sets, each of which is contractible in $X$. 
We have the following general upper bounds:	
\[
\cat(X) \leq \dim(X) \quad\text{and}\quad\TC(X) \leq 2\cat (X)\leq {2\dim(X)}. 
\]
We say that $\TC(X)$ is maximal when $\TC(X)=2\dim(X)$. This maximal value can only be reached when the fundamental group of $X$ is not trivial since $\pi_1(X)=0$ implies $\cat(X)\leq \dim(X)/2$ and consequently $\TC(X) \leq \dim(X)$.

There are numerous spaces for which $\TC$ is known to be maximal. Examples include 
surfaces of high genus
(\cite{Far}, \cite{Dranishnikov}, \cite{Dranishnikov2}, \cite{CV}, \cite{IST}) and some classes of higher dimensional manifolds and aspherical spaces. For instance, M. Grant and S. Mescher have recently established the maximality of $\TC$ for a large class of cohomologically symplectic manifolds \cite{GM}, while A. Dranisnikov has shown that, except for the circle $S^1$, any aspherical space with hyperbolic fundamental group has maximal $\TC$ \cite{Dranishnikov-hyperbolic}.

In the opposite direction, $\TC$ is not maximal for the familiar examples of manifolds with abelian fundamental group such as the $n$-torus $T^n=(S^1)^n$ \cite{Far}, the projective space $\R\P^n$ \cite{FTY}, the lens spaces $L_p^{2n+1}$ \cite{FG}, and their products.

This paper is motivated by the question to which extent having an abelian fundamental group implies the nonmaximality of the topological complexity. A first result in this direction was established by A. Costa and M. Farber:
\begin{thm}[\cite{CostaFarber}] Let $X$ be a finite CW-complex of dimension $n$.
	\begin{probs}
		\item If $\pi_1(X)=\Z_2$, then $\TC(X) <2n$.
		\item If $\pi_1(X)=\Z_3$ and either $n$ is odd, or $n=2m$ is even and the
		$3$-adic expansion of $m$ contains at least one digit $2$, then $\TC(X) <2n$.
	\end{probs}
\end{thm}
Costa-Farber also show that, when  $\pi_1(X)=\Z_3$, the condition stated in the even-dimensional case is sharp. For instance, the topological complexity of the $6$-dimensional skeleton of the lens space $L_3^{\infty}$ is maximal. Note that this space is not a closed manifold. 
Restricting our study to closed manifolds with abelian fundamental group, we establish the following generalization of Costa-Farber's theorem: 
 
\begin{thm} \label{thm:TC} Let $M$ be a $n$-dimensional connected closed manifold with $n\geq 2$.
We have $\TC(M)< 2n$ when
\begin{probs}
\item $M$ is non-orientable and $\pi_1(M)$ is abelian;
\item $M$ is orientable and $\pi_1(M)$ is an abelian group of one of the following forms (where $p$ is a prime, $r$ is a nonnegative integer, and $a,b,c,s$ are positive integers):
\begin{enumerate}
\item[(a)] $\Z^r$ with either $n$ odd,  or with $n$ even such that $2n>r$;
\item[(b)] $\Z^r\times \Z_{p^ a}$ with either $n$ odd such that $n>r$, or 
with $n$ even such that $n\geq r$;
\item[(c)] $\Z^r\times \Z_{p^ a}\times \Z_{p^ b}$ with $r\leq 1$;
\item[(d)] $\Z_{p^ a}\times \Z_{p^ b}\times \Z_{p^c}$;
\item[(e)] $\Z^r \times (\Z_2)^s$ with either $n$ odd,  or with $n$ even such that $2n>r$.
\end{enumerate} 
\end{probs}
\end{thm}

Note that the theorem is also clearly true when $n=1$, since here 
we can identify $M$ with $S^1$ and we then have $\TC(M)=1<2\dim(M)$. The sufficient conditions which appear in the orientable case are not necessary but are sharp. In Section~\ref{sec:examples}, we construct examples of manifolds with abelian fundamental group and maximal $\TC$ which show the sharpness of our conditions. For instance, Condition (2a) is sharp because Example \ref{ex:intro} exhibits a $4$-dimensional manifold $M$ with 
$\pi_1(M)=\Z^8$ and $\TC$ maximal. We note also that our statement only 
considers abelian groups which are products of infinite cyclic groups and 
primary cyclic groups for the same prime. At the end of Section \ref{subsection:statement} we briefly discuss the case of abelian groups involving 
products of cyclic groups with pairwise relatively prime orders.

We also obtain results on the LS-category of the cofibre of the diagonal map $C_{\Delta}(X)=(X\times X) /\Delta (X)$. This number is also bounded above by $2\dim(X)$ and is known to coincide with $\TC(X)$ for some classes of spaces, including spheres, orientable surfaces, real and complex projective spaces, lens spaces,
$H$ spaces \cite{GCV}, and certain two-cell complexes \cite{GGV}. However, as shown by Dranishnikov, $\cat(C_{\Delta}(X))$ and $\TC(X)$ may differ.
In particular, if $X=N_g$ is a nonorientable surface of genus $g\ge 2$, one has
$3=\cat(C_{\Delta}(N_g)) < \TC(N_g)=4=2\dim(N_g)$, see \cite{Dranishnikov}, \cite{CV}.
We show that this phenomenon -- $\cat(C_{\Delta}(X))$ nonmaximal and $\TC(X)$ maximal -- cannot occur when $\pi_1(X)$ is abelian, see Proposition \ref{prop:TCcofibre}. We also establish the nonmaximality of $\cat(C_{\Delta}(X))$ for the following class of manifolds.

\begin{thm} \label{thm:cofibre} Let $M$ be a $n$-dimensional connected closed manifold with $n\geq 2$.
We have $\cat(C_{\Delta}(M))< 2 \dim(M)$ when
\begin{probs}
\item $M$ is non-orientable;
\item $M$ is orientable and $H_1(M)$ is of of one of the following forms (where $p$ is a prime, $r$ is a nonnegative integer, and $a,b,c,s$ are positive integers):
\begin{enumerate}
\item[(a)] $\Z^r$ with either $n$ odd, or with $n$ even such that $2n>r$;
\item[(b)] $\Z^r\times \Z_{p^ a}$ with either $n$ odd such that $n>r$, or 
with $n$ even such that $n\geq r$;
\item[(c)] $\Z^r\times \Z_{p^ a}\times \Z_{p^ b}$ with $r\leq 1$;
\item[(d)] $\Z_{p^ a}\times \Z_{p^ b}\times \Z_{p^c}$;
\item[(e)] $\Z^r \times (\Z_2)^s$ with either $n$ odd, or with $n$ even such that $2n>r$.
\end{enumerate} 
\end{probs}
\end{thm}
 Again the conditions in the orientable case are not necessary but are sharp. Note that these conditions are now on $H_1(M)$ which is always abelian and which, as established in \cite[Proposition 3.4.1]{Dranishnikov}, coincides with the fundamental group of $C_{\Delta}(M)$. 
Note that we cannot relax the hypothesis for $M$ to be a manifold since, as mentioned before, the $6$-dimensional skeleton of the lens space $L_3^{\infty}$, $L_3^{(6)}$, satisfies $\TC(L_3^{(6)})=12$ and it is shown in \cite{GCV} that $\TC(L_3^{(6)})=\cat(C_{\Delta}(L_3^{(6)}))$.
We note also that, when $n=1$ and hence $M=S^1$, the conclusion of Theorem \ref{thm:cofibre} is also true since $\cat(C_{\Delta}(S^1))=\TC(S^1)=1$ \cite{GCV}. 

As an extension of the phenomenon which occurs for the nonorientable surfaces of genus $g\geq 2$, we can deduce from Theorem \ref{thm:cofibre} and a standard cup-length argument that, for any $n\geq 3$ and $g\geq 2$, the connected sum of $g$ copies of $\R\P^n$, $\XX^n_g=\R\P^n\#\cdots \# \R\P^n$, satisfies $\cat(C_{\Delta}(\XX^n_g))=2n-1$ while, as established in \cite{CV2}, $\TC(\XX^n_g)=2n$. This strongly differs from the case $g=1$ since, as mentioned before, $\TC(\R\P^n)$ is not maximal and $\TC(\R\P^n)=\cat(C_{\Delta}(\R\P^n))$. 

Theorems \ref{thm:TC} and \ref{thm:cofibre} will follow from an algebraic result established in Section \ref{sec:abelian} (Proposition \ref{prop-general}) together with some criteria for the nonmaximality of $\TC(M)$ and $\cat(C_{\Delta}(M))$ (Propositions \ref{prop:TC:abelian},\ref{prop:cofibre} and \ref{prop:equivalence}). These criteria, which are in the same spirit as Dranishnikov's approach to $\cat(C_{\Delta}(M))$ in \cite{Dranishnikov} (Proposition \ref{prop:cofibre} can actually be deduced from \cite{Dranishnikov}), are developed in Section \ref{sec:criteria} through a study of the $\TC$-canonical class introduced in \cite{CostaFarber} (which we call the Costa-Farber class) and of the classical Berstein-Schwarz class. Finally,  
we construct in Section \ref{sec:examples} examples which show that our conditions in Theorems  \ref{thm:TC} and \ref{thm:cofibre} are sharp.

\section{Nonmaximality criteria} \label{sec:criteria}
We first recall some material on the Berstein-Schwarz and Costa-Farber classes which play important roles in the study of the LS-category and topological complexity of spaces with nontrivial fundamental group. References include  \cite{Berstein}, \cite{CostaFarber}, \cite{CLOT}, \cite{DR}, \cite{Sch}. 

\subsection{Berstein-Schwarz and Costa-Farber classes} 
Let $X$ be a path-connected CW-complex and $\pi=\pi_1(X)$ its fundamental group.  Let $\Z[\pi]$ be the integral group ring and let $I(\pi)=\ker(\varepsilon\colon\Z[\pi]\to \Z)$ be the augmentation ideal. Recall that $I(\pi)$ is both a  (left) $\Z[\pi]$- and $\Z[\pi\times \pi]$-module through the actions given by:
\[
a\cdot \sum n_ia_i= \sum n_i aa_i \quad (a, b) \cdot \sum n_ia_i= \sum n_i (aa_i\bar{b}).
\]
Here $n_i\in \Z$, $a,b,a_i\in \pi$ and $\bar{a}$ denotes the inverse of $a$. The ideal $I(\pi)$ can then be regarded as a system of local coefficients over both $X$ and $X\times X$ and we will use the notations $\mathrm{I}[\pi]$ and $\varmathbb{I}[\pi]$ to distinguish the two structures. We refer to  \cite[VI]{Wh} and \cite[Chapter 5]{DavisKirk} as general references for local coefficient systems and to \cite{Brown} for the special case of group (co)homology.  If $P$ and $Q$ are (left) $\Z[\pi]$-modules we denote by  $P\otimes_{\pi}Q$ the coinvariants of the (left) $\Z[\pi]$-module $P\otimes Q$ given with the diagonal action of $\pi$, and by $P^k$ the $k$-fold tensor power of $P$ with the respective diagonal action. In particular, $\mathrm{I}^k[\pi]$ and $\varmathbb{I}^k[\pi]$ denote the $k$-fold tensor powers of $\mathrm{I}[\pi]$ and $\varmathbb{I}[\pi]$.

The Berstein-Schwarz class ${\mathfrak b}_X\in H^1(X;\mathrm{I}[\pi])$ and the Costa-Farber class  ${\mathfrak{v}}_X\in H^ 1(X\times X;\varmathbb{I}[\pi])$ can be defined as the cohomology classes induced by the following crossed homomorphisms respectively:
\[
\begin{array}{rcl}
\pi & \to & \mathrm{I}[\pi]\\
a&\mapsto &a-1
\end{array}
\qquad
\begin{array}{rcl}
\pi\times \pi & \to & \varmathbb{I}[\pi]\\
(a, b)&\mapsto &a\bar{b}-1.
\end{array}
\]

When $X=B\pi$ is the classifying space of $\pi$, the (co)homology groups of $B\pi$ and $B\pi\times B\pi=B(\pi\times \pi)$ can be identified to the (co)homology groups of $\pi$ and $\pi\times \pi$ and the classes ${\mathfrak b}_{\pi}:={\mathfrak b}_{B\pi}$ and ${\mathfrak v}_{\pi}:={\mathfrak v}_{B\pi}$ are the classes corresponding to the following degree 1 cocycles defined on the associated bar resolutions by formulas as above:
\[
\begin{array}{rcl}
\beta_{\pi}:\mathrm{Bar}_1(\pi) & \to & \mathrm{I}[\pi]\\
\left[a\right]  &\mapsto &a-1
\end{array}
\qquad
\begin{array}{rcl}
\nu_{\pi}:\mathrm{Bar}_1(\pi\times \pi) & \to & \varmathbb{I}[\pi]\\
\left[(a, b)\right] &\mapsto &a\bar{b}-1.
\end{array}
\]

In the general case, if $\gamma:X\to B\pi$ is a map inducing an isomorphism of fundamental groups, then, through isomorphisms $\mathrm{I}[\pi]\cong \gamma^*\mathrm{I}[\pi]$ and $\varmathbb{I}[\pi]\cong (\gamma\times \gamma)^*\varmathbb{I}[\pi]$, we have the following identifications:
\[
{\mathfrak b}_X=\gamma^*{\mathfrak b}_{\pi} \qquad {\mathfrak v}_X=(\gamma\times \gamma)^*{\mathfrak v}_{\pi}. 
\]
For any $k\geq 1$, we have:
\begin{itemize}
	\item If ${\mathfrak b}^k_X\in H^k(X;\mathrm{I}^k[\pi])$ does not vanish, then $\cat(X)\geq k$.
		\item If ${\mathfrak v}^k_X\in H^k(X\times X;\varmathbb{I}^k[\pi])$ does not vanish, then $\TC(X)\geq k$.
\end{itemize}
Furthermore, we have: 

\begin{thm} \label{BS-CF}  Suppose that $\dim(X)=n\geq 2$. 
\begin{probs}
	\item \textup{(\cite{Berstein}, \cite{Sch}, \cite{DR})} \quad $\cat(X)=n$ if and only 
if ${\mathfrak b}^n_X\in H^n(X;\mathrm{I}^n[\pi])$ does not vanish.
\item \textup{(\cite{CostaFarber})} \quad $\TC(X)=2n$ if and only if ${\mathfrak v}^{2n}_X\in H^{2n}(X\times X;\varmathbb{I}^{2n}[\pi])$ does not vanish.
\end{probs}	
\end{thm}

\noindent \textbf{When $\pi$ is abelian.} 
In this case, $\chi: \pi \times \pi \to \pi$, $(a,b)\mapsto a\bar{b}$ is a homomorphism which satisfies $\chi(a,b)\cdot z=(a,b) \cdot z$ for any $(a,b)\in \pi \times \pi$, $z \in I(\pi)$. In other words, the $\Z[\pi\times \pi]$-module $\chi^*(\mathrm{I}[\pi])$ is exactly the $\Z[\pi\times \pi]$- module $\varmathbb{I}[\pi]$. With the notation above,  we also have, for any $(a,b)\in \pi \times \pi$,
\[\nu_{\pi}([(a,b)])=\beta_\pi([a\bar{b}])=\beta_{\pi}([\chi(a,b)]).\]
We then have ${\mathfrak v}_{\pi}=\chi^*{\mathfrak b}_{\pi}$ in $H^1(\pi\times \pi;\varmathbb{I}[\pi])=H^1(B\pi\times B\pi;\varmathbb{I}[\pi])$ (see also \cite{FM}), and, for any $k$,
\[{\mathfrak v}^k_X= (\gamma\times \gamma)^*{\mathfrak v}^k_{\pi}= (\gamma\times \gamma)^*\chi^*{\mathfrak b}^k_{\pi} \,\, \mbox{ in } H^k(X\times X;\varmathbb{I}^k[\pi]).\]

\subsection{Nonmaximality condition for \texorpdfstring{$\TC(M)$}{TCM}}
We first consider an $n$-dimensional closed manifold $M$ with $\pi=\pi_1(M)$ not necessarily abelian. Let $\widetilde{\Z}$ be the $\Z[\pi]$ orientation module of $M$, that is, $\widetilde{\Z}$ represents the integers with an action of $\pi$ given by $a\cdot t=\omega(a)t$ ($a\in \pi$, $t$ an integer) where $\omega=\omega_M:\pi \to \{\pm 1\}$ is the homomorphism determined by the first Stiefel-Whitney class of $M$. The orientation module of a product $M\times N$ is given by the action $(a,b)\cdot t=\omega_M(a)\omega_N(b) t$. 

Let $[M]\in H_n(M;  \widetilde{\Z})\cong \Z$ be the (twisted) fundamental class of $M$. Using the K\"unneth formula (for twisted homology, see for instance \cite{G}), we can see that the fundamental class of $M\times M$, 
$[M\times M]\in H_{2n}(M\times M; \widetilde{\Z})=\Z$, can be identified with the cross product $[M]\times [M]$.

By Poincar\'e duality, Theorem \ref{BS-CF}(2) implies that $\TC(M)=2n$ if and only if \[[M\times M] \cap {\mathfrak v}^{2n}_M \in H_0(M\times M;\varmathbb{I}^{2n}[\pi]\otimes \widetilde{\Z})= \varmathbb{I}^{2n}[\pi]\otimes_{\pi\times \pi} \widetilde{\Z}\] does not vanish. Considering as before a map $\gamma:M\to B\pi$ inducing an isomorphism of fundamental groups, we equip $\widetilde{\Z}$ with the $\pi$ action induced by the isomorphism $\pi_1(\gamma)$,  making $\gamma$ compatible with the action. In particular, $\gamma_*:H_*(M; \widetilde{\Z}) \to H_*(B\pi; \widetilde{\Z})$ is well-defined. The same is done for the products $M\times M$ and $\gamma\times \gamma$. Then, through
the identification  ${\mathfrak v}_X=(\gamma\times \gamma)^*{\mathfrak v}_{\pi}$ and the naturality of the cap-product given by the diagram
\begin{equation*}
\xymatrixcolsep{5pc}
\xymatrix{
	H_{2n}(M\times M;\widetilde{\Z}) \otimes H^{2n}(M\times M;\varmathbb{I}^{2n}[\pi]) \ar@<-8ex>[d]^{(\gamma\times \gamma)_*} \ar[r]^-{\cap} & 
	\varmathbb{I}^{2n}[\pi]\otimes_{\pi\times \pi}\widetilde{\Z} \ar[d]^{=} \\
	H_{2n}(B\pi \times B\pi;\widetilde{\Z}) \otimes H^{2n}(B\pi \times B\pi;\varmathbb{I}^{2n}[\pi])  \ar@<-8ex>[u]_{(\gamma\times \gamma)^*} 
	\ar[r]^-{\cap} & 
	\varmathbb{I}^{2n}[\pi]\otimes_{\pi\times \pi}\widetilde{\Z}
}
\end{equation*}
we see that $\TC(M)=2n$ if and only if $(\gamma \times \gamma)_*[M\times M] \cap {\mathfrak v}^{2n}_{\pi}$ does not vanish or, equivalently, if and only if the following composite is not trivial.
\begin{equation}  \label{eq:composite} \tag{$\dag$}
\Z=H_{2n}(M\times M;\widetilde{\Z}) \xrightarrow{(\gamma\times \gamma)_*}  H_{2n}(B\pi\times B\pi; \widetilde{\Z}) \xrightarrow{\cap \,\mathfrak{v}_{\pi}^{2n} }  \varmathbb{I}^{2n}[\pi]\otimes_{\pi\times \pi} \widetilde{\Z}.
\end{equation}
In what follows, we will use the notation \[\mathfrak{m}=\gamma_*([M])\in H_n(B\pi; \widetilde{\Z})=H_n(\pi; \widetilde{\Z}).\] 
We then have
\[\mathfrak{m}\times \mathfrak{m}=\gamma_*([M])\times \gamma_*([M])=(\gamma\times \gamma)_*([M\times M])\in H_{2n}(B\pi\times B\pi;\widetilde{\Z})=H_{2n}(\pi\times \pi;\widetilde{\Z}).\]

\noindent \textbf{When $X=M$ is a closed manifold with $\pi=\pi_1(M)$ 
abelian.} In this case, the homomorphism $\chi: \pi \times \pi \to \pi$, $(a,b)\mapsto a\bar{b}$ induces a map $B\pi\times B\pi=B(\pi\times \pi) \to B\pi$, which will also be denoted by $\chi$. This map is compatible with any fixed $\Z[\pi]$ orientation module since, for any $\omega:\pi \to \{\pm 1\}$ and any $(a,b)\in \pi \times \pi$, one has $\omega(a\bar b)=\omega(a)\omega(b)$. Consequently $\chi$ induces a map $\chi_*: H_*(B\pi\times B\pi; \widetilde{\Z})\to H_*(B\pi; \widetilde{\Z})$ and we have:

\begin{prop} \label{prop:TC:abelian} Let $M$ be an $n$-dimensional connected closed manifold with abelian fundamental group $\pi=\pi_1(M)$. If $\chi_*(\mathfrak{m}\times \mathfrak{m})=0$ in $H_{2n}(B\pi; \widetilde{\Z})$ then $\TC(M)< 2n$.
\end{prop} 

\begin{proof}
	From the considerations above and the equality ${\mathfrak v}^{2n}_{\pi}= \chi^*{\mathfrak b}^{2n}_{\pi}$ we can form the following commutative diagram
\[	 \xymatrix{
		\Z=H_{2n}(M\times M;\widetilde{\Z}) \ar[r]^-{(\gamma\times \gamma)_*} & H_{2n}(B\pi\times B\pi; \widetilde{\Z}) \ar[d]_{\chi_*} \ar[r]^-{\cap \,\mathfrak{v}_{\pi}^{2n} } &\varmathbb{I}^{2n}[\pi]\otimes_{\pi\times \pi} \widetilde{\Z}\ar[d]_{}^{\cong}\\
		& H_{2n}(B\pi; \widetilde{\Z}) \ar[r]^{\cap \,\mathfrak{b}_{\pi}^{2n} } & \mathrm{I}^{2n}[\pi]\otimes_{\pi} \widetilde{\Z}.	}
\]
The right-hand vertical map corresponds to the map $\chi_*: H_*(B\pi\times B\pi; \varmathbb{I}^{2n}[\pi]\otimes \widetilde{\Z})\to H_*(B\pi; \mathrm{I}^{2n}[\pi]\otimes\widetilde{\Z})$ in degree $0$. It is induced by the identity of the underlying $\Z$-module $I^{2n}(\pi)\otimes \widetilde{\Z}$ and is an isomorphism on the coinvariants because $\chi$ is surjective. 

Since $[M\times M]$ is a generator of $\Z=H_{2n}(M\times M;\widetilde{\Z})$ we see that $\chi_*(\mathfrak{m}\times \mathfrak{m})=0$ implies that the composite in the top line of the diagram is trivial which in turn implies that $\TC(M)<2n$, see \eqref{eq:composite} and the surrounding discussion.
	\end{proof}

As indicated in the introduction, Theorem \ref{thm:TC} will follow from Proposition \ref{prop:TC:abelian} together with Proposition \ref{prop-general} which is established in Section 3. We note that the converse of Proposition \ref{prop:TC:abelian} is also true, see Proposition \ref{prop:equivalence}. 

\subsection{LS-category of the cofibre of the diagonal map.} \label{sec:cofibre}
Let $X$ be a path-connected CW complex of dimension $n$ and let $C_{\Delta}(X)=(X\times X) /\Delta (X)$ be the cofibre of the diagonal map. As $\dim(C_{\Delta}(X))=2n$ we have $\cat(C_{\Delta}(X))\leq 2n$.
As is the case for Theorem \ref{thm:TC}, Theorem \ref{thm:cofibre} will follow from Proposition \ref{prop-general} together with a nonmaximality condition for $\cat(C_{\Delta}(X))$ when $X=M$ is a manifold, namely Proposition \ref{prop:cofibre}. As mentioned before, this criterion can be deduced from Dranishnikov's approach in \cite{Dranishnikov}. Here we organize the argument in order to emphasize the interrelations between the Berstein-Schwarz and Costa-Farber classes. We also show that for a CW-complex with abelian fundamental group, $\TC$ is maximal if and only if the LS-category of the cofibre of the diagonal map is maximal.

We first consider a general $n$-dimensional path-connected CW-complex with fundamental group $\pi=\pi_1(X)$. 
Let $\gamma:X\to B\pi$ be a map which induces an isomorphism of fundamental groups. Denote by $\alpha: \pi=\pi_1(X)\to H=H_1(X)$ the abelianization homomorphism and by $\chi:H\times H\to H$ the homomorphism given by $\chi(g,h)=g\bar h$.  As before, we will use the same symbols for 
the induced maps $B\pi\to BH$ and $BH\times BH \to BH$. We will also consider the morphism $I(\alpha):I(\pi)\to I(H)$ and will denote by $\varmathbb{A}$ the augmentation ideal $I(H)$ equipped with the $\pi \times \pi$ action induced by $\alpha\times \alpha$. Explicitly, for $(a,b) \in \pi \times \pi$ and $\sum n_i h_i$ in $I(H)$, we have 
\[(a,b)\cdot \sum n_i h_i:=(\alpha(a),\alpha(b))\cdot \sum n_i h_i=\sum n_i \alpha(a) h_i \overline{\alpha(b)}=\chi(\alpha(a), \alpha({b}))\sum n_i h_i\]
so that $\varmathbb{A}=(\alpha\times \alpha)^*\chi^* \mathrm{I}[H]$.
Through the identification $(\gamma \times \gamma)^*\varmathbb{A} \cong \varmathbb{A}$ we consider $\varmathbb{A}$ as a system of local coefficients on $X\times X$. The restriction to $\Delta(X)$ of the composite
\[X\times X \xrightarrow{\gamma\times \gamma}B\pi \times B\pi \xrightarrow{\alpha\times \alpha}BH \times BH \xrightarrow{{\chi}} BH\]
is trivial and factors through $C_{\Delta}(X)$ giving a commutative diagram
\begin{equation}\label{diagram:cofibre}
\tag{${\dag\dag}$}
\xymatrix{
	X\times X \ar[d]^-{q}\ar[r]^{ \gamma \times \gamma\ } & B\pi\times B\pi \ar[r]^{\alpha \times \alpha\ }& BH\times BH \ar[d]^{\chi}\\
	C_{\Delta}(X) \ar[rr]^-{\xi} &  &BH.
	&&
}
\end{equation}

\begin{lem} \label{lemma:cofibre} With notation as above, we have 
	\begin{probs}
	\item $\mathfrak{b}_{C_{\Delta}(X)}=\xi^* \mathfrak{b}_H$.
		\item $I(\alpha) \mathfrak{v}_X=q^*\mathfrak{b}_{C_{\Delta}(X)}$ in $H^1(X\times X; \varmathbb{A})$.
		\item $\forall k \geq n+2$, $I^k(\alpha) \mathfrak{v}^k_X=0 \iff q^*\mathfrak{b}^k_{C_{\Delta}(X)}=0 \iff \mathfrak{b}^k_{C_{\Delta}(X)}=0$.
	\end{probs}
	\end{lem}

\begin{proof}
	As shown in \cite[Proposition 3.12]{Dranishnikov}, the fundamental group 
of $C_{\Delta}(X)$ is isomorphic to $H$. Moreover the homomorphism induced by the map $q:X\times X \to C_{\Delta}(X)$ can be identified (up to an isomorphism) with the homomorphism
	\[\pi\times \pi \xrightarrow{\alpha\times \alpha}H\times H \xrightarrow{\chi} H\] and $\xi$ induces an isomorphism of fundamental groups.
	This corresponds to the fact that the diagram
	$$\xymatrix{
		\pi \ar[d]\ar[r]^{\Delta} & \pi\times \pi \ar[d]^{\chi(\alpha\times \alpha)}\\
		1 \ar[r] & H=\pi/[\pi,\pi]
	}$$
	is a push-out of groups. Thus, $\mathfrak{b}_{C_{\Delta}(X)}=\xi^* \mathfrak{b}_H$ as asserted in the first statement. 
	
	Through the identification $\mathrm{I}[H]=\xi^*\mathrm{I}[H]$, we also see that $q^*(\mathrm{I}[H])=\varmathbb{A}$ and  that $q$ induces a morphism $$q^*: H^k(C_{\Delta}(X); \mathrm{I}^k[H]) \to H^k(X\times X; \varmathbb{A}^k).$$
	From the diagram \eqref{diagram:cofibre}, we thus have $q^*\mathfrak{b}_{C_{\Delta}(X)}= (\gamma \times \gamma)^*(\alpha\times \alpha)^* \chi^* \mathfrak{b}_H$. Since $H$ is abelian, we have $\chi^* \mathfrak{b}_H=\mathfrak{v}_H$. The second statement then follows from the equality $(\alpha\times \alpha)^*\mathfrak{v}_H=I(\alpha)\mathfrak{v}_{\pi}$ and the identification $I(\alpha)\mathfrak{v}_{X}=(\gamma\times \gamma)^*(I(\alpha)\mathfrak{v}_{\pi})$.
	
	For the third statement, we consider, for $k\geq n+2$, the long exact sequence in cohomology of the cofibration $X\stackrel{\Delta}{\to} X\times X \stackrel{q}{\to} C_{\Delta}(X)$:
	\[\to H^{k-1}(X;\varmathbb{A}_{|X}^ {k})\stackrel{\delta}{\to} H^{k}(C_{\Delta}(X);\mathrm{I}^ {k}[H])\stackrel{q^ *}{\to} H^{k}(X\times X;\varmathbb{A}^{k})\stackrel{\Delta^ *}{\to}H^{k}(X;\varmathbb{A}_{|X}^ {k})\stackrel{\delta}{\to} \cdots\]
		The local system  $\varmathbb{A}_{|X}$ on $X$ induced by $\Delta$ is trivial. But notice that the long exact sequence need not split in general because the possible retractions of $\Delta$ are not necessarily compatible with the structure of the local coefficients. However, since $\dim X=n$ and $k\geq n+2$, we have $H^{k-1}(X; \varmathbb{A}_{|X}^{ k})=H^{k}(X; \varmathbb{A}_{|X}^{ k})=0$ so that 
	\[q^*:H^{k}(C_{\Delta}(X); \mathrm{I}^{k}[H])\to H^{k}(X\times X; \varmathbb{A}^{k})\]
	is an isomorphism and the third statement follows from the first statement.
	\end{proof}

As a consequence, we have:

\begin{prop}\label{prop:TCcofibre}
	Let $X$ be a path-connected CW-complex of dimension $n\geq 2$. If $\cat(C_{\Delta}(X)))=2n$ then $\TC(X)=2n$. If, moreover $\pi_1(X)$ is abelian, then the converse is true.
\end{prop}

\begin{proof}
	The space $C_{\Delta}(X)$ has dimension $2n$. By Theorem \ref{BS-CF}, $\cat(C_{\Delta}(X))=2n$ implies that $\mathfrak{b}_{C_{\Delta}(X)}^{2n}\neq 0$. Since $n\geq 2$, Lemma \ref*{lemma:cofibre} implies that ${\mathfrak v}^{2n}_X\neq 0$. Thus, by Theorem \ref{BS-CF}, $\TC(X)=2n$. If $\pi_1(X)$ is abelian, then the homomorphism $\alpha$ is the identity and 
the third statement of Lemma \ref*{lemma:cofibre} yields ${\mathfrak v}^{2n}_X\neq 0$ if and only if $\mathfrak{b}_{C_{\Delta}(X)}^{2n}\neq 0$.
\end{proof}

We now suppose that $X=M$ is a manifold of dimension $n\geq 2$.
 As before $\mathfrak{m}=\gamma_*([M])\in H_n(\pi; \widetilde{\Z})$. Note that  $\alpha_*(\mathfrak{m})\in H_n(H;\widetilde{\Z})$ where the action of $H$ on $\widetilde{\Z}$ is induced by $\alpha$ and is well-defined because the kernel of $\alpha$, namely $[\pi,\pi]$, acts trivially on $\widetilde{\Z}$. Dranishnikov's nonmaximality condition for $\cat(C_{\Delta}(M))$ can then be stated as:

\begin{prop}[\cite{Dranishnikov}] \label{prop:cofibre} Let $M$ be an $n$-dimensional connected closed manifold. We have $\cat(C_{\Delta}(M))\leq 2n-1$ if and only if $\chi_*(\alpha_*(\mathfrak{m})\times \alpha_*(\mathfrak{m}))=0$ in $H_{2n}(H;\widetilde{\Z})$.
\end{prop}

\begin{proof}
	\noindent 
Using Lemma \ref{lemma:cofibre}, our proof essentially follows the arguments of \cite[Proposition 2.4.2]{Dranishnikov} without requiring the notion of pseudo-manifolds. We have
	$$ \cat(C_{\Delta}(M))\leq 2n-1 \iff q^*\mathfrak{b}_{C_{\Delta}(M)}^{2n}=0 \mbox{ in } H^{2n}(M\times M, \varmathbb{A}^{2n}).$$
	Using the equality $\mathfrak{b}_{C_{\Delta}(M)}=\xi^*\mathfrak{b}_H$ and Poincar\'e duality for $M\times M$, this continues as 
	\[
	\cat(C_{\Delta}(M))\leq 2n-1 \iff  [M\times M] \cap (q^*\xi^*(\mathfrak{b}_H))^{2n}=0 \,\,\mbox{in } H_0(M\times M, \varmathbb{A}^{2n}\otimes \widetilde{\Z}).\]
	The naturality of cap-products yields the diagram
	\begin{equation*}
	\xymatrixcolsep{5pc}
	\xymatrix{
		H_{2n}(M\times M;\widetilde{\Z}) \otimes H^{2n}(M\times M;\varmathbb{A}^{2n}) \ar@<-8ex>[d]^{\xi_* q_*} \ar[r]^-{\cap} & 
		\varmathbb{A}^{2n}\otimes_{\pi\times \pi}\widetilde{\Z} \ar[d]^{\cong} \\
		H_{2n}(BH;\widetilde{\Z}) \otimes H^{2n}(BH;\mathrm{I}^{2n}[H])  \ar@<-8ex>[u]_{q^*\xi^*} 
		\ar[r]^-{\cap} & 
		\mathrm{I}^{2n}[H]\otimes_{H}\widetilde{\Z}
	}
	\end{equation*}
	The righthand vertical map is induced by the identity of the underlying $\Z$-module $I^{2n}(H)\otimes \widetilde{\Z}$ and is an isomorphism on the coinvariants because the homomorphism $\xi_*q_*\cong \chi(\alpha\times \alpha):\pi\times \pi \to H$ is surjective. We then have
\[\cat(C_{\Delta}(M))\leq 2n-1\iff  \xi_*q_*[M\times M]\cap \mathfrak{b}_H^ {2n}=0.\]
	
	Since $ \xi_*q_*[M\times M]=\chi_*(\alpha_*{\mathfrak m}\times \alpha_*{\mathfrak m})$, we can conclude that the vanishing of this class implies $\cat(C_{\Delta}(M))\leq 2n-1$. Conversely, if $\cat(C_{\Delta}(M))\leq 2n-1$, then the map $\xi$, and consequently $\xi q$, can be factored up to homotopy through an $(2n-1)$ dimensional space as, for instance, in \cite[Proposition 4.3]{Dranishnikov-Lens}). As this factorization can be made in a compatible way with the action on $\widetilde{\Z}$, we obtain $\xi_*q_*[M\times M]=0$ and therefore $\chi_*(\alpha_*{\mathfrak m}\times \alpha_*{\mathfrak m})=0$.
\end{proof}

We conclude this section by noting the following.

\begin{prop}\label{prop:equivalence}
	Let $M$ be an $n$-dimensional connected closed manifold. If $\pi_1(M)$ is abelian, then the following are equivalent:
	\begin{itemize}
		\item $\TC(M)=2n$
		\item $\cat(C_{\Delta}(M))=2n$
		\item $\chi_*(\mathfrak{m}\times \mathfrak{m})\neq 0$ in $H_{2n}(\pi_1(M);\widetilde{\Z})$.
	\end{itemize}
\end{prop}

\section{Calculations in  \texorpdfstring{$H_*(G;\widetilde{\Z})$}{H} for 
\texorpdfstring{$G$}{G} abelian}  \label{sec:abelian}

\subsection{Main statement}\label{subsection:statement}
Let $G$ be a finitely generated abelian group and $\chi:G\times G \to G$ the homomorphism $(g,h)\mapsto g\bar{h}$. We assume $G$ is given with an homomorphism $\omega:G\to \{\pm 1\}$ and consider the associated $\Z[G]$- and $\Z[G\times G]$-modules $\widetilde{\Z}$. As before, the action of $G\times G$ on an integer $t$ is given by $(g,h)\cdot t=\omega(g)\omega(h)t$ and the $\Z[G\times G]$-module $\widetilde{\Z}$ is isomorphic to the $\Z[G]\otimes \Z[G]$-module $\widetilde{\Z}\otimes\widetilde{\Z}$ where the action is componentwise. The homomorphism $\chi$ satisfies $\chi=\mu({\rm Id} \times j)$, where $\mu$ is the multiplication and $j:G\to G$, $g\mapsto \bar g$, is the inversion. Since $\omega(gh)=\omega(g)\omega(h)$ and $\omega(\bar g)=\omega(g)$, all these homomorphisms are compatible with the action on $\widetilde{\Z}$. 

Let $j_*$ and $\chi_*$ denote the morphisms induced by $j$ and $\chi$ in homology with coefficients in $\widetilde{\Z}$,  and denote by $\wedge$ the Pontryagin product $H_{*}(G;\widetilde{\Z})\otimes H_{*}(G;\widetilde{\Z})\to H_{*}(G;\widetilde{\Z})$ induced by $\mu$. 
Independent of the action being trivial or not, $\wedge$ is strictly anticommutative, that is, $\mathfrak{a}\wedge \mathfrak{b}=(-1)^{|\mathfrak{a}||\mathfrak{b}|}\mathfrak{b}\wedge \mathfrak{a}$ (where $|\cdot|$ denotes the degree) and $\mathfrak{a}\wedge \mathfrak{a}=0$ when $|\mathfrak{a}|$ is odd (see below or \cite[page 118]{Brown} where the argument is given for the trivial action case but can be readily adapted to the nontrivial action case). 

For a class $ \mathfrak{c}\in H_{*}(G;\widetilde{\Z})$, we then have 
\[\chi_*(\mathfrak{c}\times \mathfrak{c})=\mathfrak{c} \wedge j_*(\mathfrak{c}).\]
 The goal of this section is to establish the following proposition. As noted before, Theorems \ref{thm:TC} and \ref{thm:cofibre} will then follow from Propositions \ref{prop:TC:abelian} and \ref{prop:cofibre} and this result.

\begin{prop}\label{prop-general} For ${\mathfrak c} \in H_ n(G; \widetilde{\Z})$ with $n\geq 2$, we have $\chi_*(\mathfrak{c} \times \mathfrak{c}) =0$ 
when
\begin{probs}
	\item the action of $G$ on $\widetilde{\Z}$ is not trivial;
	\item the action of $G$ on $\widetilde{\Z}$ is trivial and $G$ is of one 
of the following forms (where $p$ is a prime, $r$ is a nonnegative integer, and $a,b,c,s$ are positive integers):
	\begin{enumerate}
		\item[(a)] $\Z^r$ with either $n$ odd, or with $n$ even such that $2n>r$;
		\item[(b)] $\Z^r\times \Z_{p^ a}$ with either $n$ odd such that $n>r$, or with $n$ even such that $n\geq r$;
		\item[(c)] $\Z^r\times \Z_{p^ a}\times \Z_{p^ b}$ with $r\leq 1$;
		\item[(d)] $\Z_{p^ a}\times \Z_{p^ b}\times \Z_{p^c}$;
		\item[(e)] $\Z^r \times (\Z_2)^s$ with either $n$ odd,  or with $n$ even such that $2n>r$. 
	\end{enumerate} 
\end{probs}
\end{prop}
The proof of this proposition will be given in Section \ref{subsection:proof-prop-general}. Note that, if $j_*(\mathfrak{c})=\pm \mathfrak{c}$ and $|\mathfrak{c}|$ is odd, then the strict anticommutativity of $\wedge$ 
immediately implies that $\chi_*(\mathfrak{c}\times \mathfrak{c})=\mathfrak{c} \wedge j_*(\mathfrak{c})=0$. However, we may have $j_*(\mathfrak{c})\neq\pm \mathfrak{c}$ (see Example \ref{ex:cond_b}). Our calculations 
will use explicit expressions of the morphism induced by $j$ at the chain level developed in Section \ref{subsection:Pontryagin}. In the course of 
the proof of Proposition \ref{prop-general}, we will frequently use the twisted K\"unneth formula given by the following split exact sequence 
with coefficients in the $\Z[G]$-, $\Z[H]$-, and $\Z[G\times H]$-modules $\widetilde{\Z}$ temporarily suppressed.
\[
0 \rightarrow \bigoplus\limits_{i+j=n}H_i(G)\otimes H_j(H)\rightarrow
H_n(G\times H)\rightarrow
\bigoplus\limits_{i+j=n-1}{\rm Tor}(H_i(G), H_j(H))\rightarrow
0
\] 
Here, as before, the action of $(g,h)\in G\times H$ on an integer $t$ is given by $(g,h)\cdot t=\omega_G(g)\omega_H(h)t$ and the $\Z[G\times H]$-module $\widetilde{\Z}$ is isomorphic to the $\Z[G]\otimes \Z[H]$-module $\widetilde{\Z}\otimes\widetilde{\Z}$ where the action is componentwise.

We finally note that although Proposition \ref{prop-general}(2) only considers groups of the form $\Z^r\times G_p$ where $p$ is a prime and $G_p$ is a finite product of $p$-primary cyclic groups $\Z_{p^a}$, more general 
statements might be valid. For instance (2b) is valid for any product $\Z^r \times \Z_q$ with $q\geq 2$, see Proposition \ref{torusxcyclic}.  More generally, for a finite number of pairwise relatively prime $p_i$, the K\"unneth formula permits us to see that, for any $r\geq 0$, we have 
\[H_*(\Z^r\times \Pi_iG_{p_i})= H_*(\Z^r) \otimes (\bigoplus\nolimits_{\Z} H_*(G_{p_i}))\]
Here, $\bigoplus_{\Z}$ is used to indicate the sum of algebras: in degree $0$ all the $H_0=\Z$ are identified and, for $p_i\neq p_j$, $ H_+(G_{p_i}) \wedge  H_+(G_{p_j})=0$. We can then see that if $\Z^r\times G_{p_i}$ belongs to the list considered in (2) for all $i$, then we will have $\chi_*(\mathfrak{c}\times \mathfrak{c})=0$ for any $\mathfrak{c} \in H_n(\Z^r\times \Pi_iG_{p_i})$ where $n\geq 2$.
 
\subsection{Pontryagin chain algebra and inversion.}\label{subsection:Pontryagin}
Recall that, given $P_{\bullet}\to \Z$ a free resolution of $\Z$ as a trivial $\Z[G]$-module,  $H_*(G;\widetilde{\Z})$ is the homology of the chain complex $P_{\bullet}\otimes_G\widetilde{\Z}$.

The Pontryagin product on $H_*(G;\widetilde{\Z})$ can be seen as induced by the shuffle product on the bar resolution $\mathrm{Bar}_{\bullet}(G)\to \Z$, that is, the $\Z[G]$-bilinear product naturally induced by the multiplication $\mu:G\times G\to G$ through the Eilenberg-Zilber equivalence $\mathrm{Bar}_{\bullet}(G)\otimes \mathrm{Bar}_{\bullet}(G)\stackrel{\sim}{\to} \mathrm{Bar}_{\bullet}(G\times G)$. This $\Z[G]$-bilinear product is compatible with the augmentation and satisfies the usual graded Leibniz rule, corresponding to what is called an \textit{admissible product} in 
\cite[page 118]{Brown}. Together with the multiplication of integers, it induces an algebra structure on the chain complex $\mathrm{Bar}_{\bullet}(G)\otimes_G\widetilde{\Z}$ which in turn induces the Pontryagin product on $H_*(G;\widetilde{\Z})$.

We will say that such a $\Z[G]$-free resolution of $\Z$, $P_{\bullet}\to \Z$, given with an admissible $\Z[G]$-bilinear product $\wedge:P_{\bullet}\otimes P_{\bullet}\to P_{\bullet}$ is a \textit{$\Z[G]$-Pontryagin chain algebra} and that the resulting $\Z$ chain algebra $P_{\bullet}\otimes_G\widetilde{\Z}$ is a $\Z$-Pontryagin chain algebra for $G$. In these terms, Example (5.4) of \cite[page 119]{Brown} means that if $A_{\bullet}$ is a $\Z[G]$-Pontryagin chain algebra (with product $\wedge_G$) and $B_{\bullet}$ is $\Z[H]$-Pontryagin chain algebra (with product $\wedge_H$) then the tensor product $A_{\bullet}\otimes B_{\bullet}$ with the product 
	\[(a\otimes b)\wedge (a'\otimes b')=(-1)^{|a'||b|}(a \wedge_G a') \otimes (b \wedge_H b')\]
	is a $\Z[G\times H]\cong \Z[G]\otimes \Z[H]$-Pontryagin chain algebra. It then follows that the tensor product 
\[A_{\bullet}\otimes_G\widetilde{\Z}\otimes B_{\bullet}\otimes_H\widetilde{\Z}\cong A_{\bullet}\otimes B_{\bullet}\otimes_{G\times H}\widetilde{\Z}\]
yields a $\Z$-Pontryagin chain algebra for $G\times H$.

To study one of the groups $G$ arising in Proposition \ref{prop-general}, we can thus fix a decomposition of $G$ as a product of infinite and primary cyclic groups and consider the tensor product of chain algebras associated to the factors. As it is more convenient for our calculations we will follow this process.
 
We start with the classical small resolutions $P_{\bullet}(G)\to \Z$ associated to the cyclic groups $G=\Z$ and $G=\Z_{q}$ recorded below. In each case, we have at most a rank $1$ free $\Z[G]$-module in degree $i$ and we denote by $[i]$ a corresponding basis element. We recall the associated $\Z[G]$-bilinear product (as described in  \cite[page 118-120]{Brown}) which is strictly anticommutative and which we denote by $\wedge$. We also make explicit an augmentation preserving morphism induced by the inversion. This is a chain map $j_{\bullet}:P_{\bullet}(G) \to P_{\bullet}(G)$ which satisfies $j_{\bullet}([0])=[0]$ and, for any $i\geq 0$ and $g\in G$, $j_{\bullet}(g[i])=\bar g j_{\bullet}([i])$.

For $G=\Z=\langle u\rangle$, we consider the resolution given by 
\[
\xymatrix{
&\Z [\Z] \ar[r]^{1-u}& \Z [\Z] \ar[r]^{\varepsilon}&\Z\\
}\]
The product and the morphism $j_{\bullet}$ are given by
\[[0]\wedge [0]=[0] \quad  [1]\wedge [0]=[0]\wedge [1]=[1] \quad [1]\wedge [1]=0\]
and 
\[j_{\bullet}([0])=[0] \quad  j_{\bullet}([1])=-\bar u[1]. \]

For $G=\Z_q=\langle v ~|~v^q=1\rangle$, we consider the resolution given by
\[
\xymatrix{
\cdots \ar[r] &\Z [\Z_q]\ar[r]^{N_q(v)} & \Z [\Z_q] \ar[r]^{1-v} & \Z [\Z_q] \ar[r]^{N_q(v)} & \Z [\Z_q] \ar[r]^{1-v}& \Z [\Z_q] \ar[r]^{\varepsilon}&\Z
}\]
where $N_q(v)=1+v+\cdots+v^{q-1}$. Here we have
\begin{equation}\label{productformula}
\tag{$\ast$}
\left\{\begin{array}{l}
\arraycolsep=1.4pt\def\arraystretch{1.5}
~[2i]\wedge [2k]=\binom{i+k}{k}[2i+2k] \quad \quad [2i+1]\wedge [2k+1]=0 \\
\\
~[2i]\wedge [2k+1]=[2k+1]\wedge [2i]=\binom{i+k}{k}[2i+2k+1]
\end{array}\right.
\end{equation}
and 
\begin{equation}\label{productj}
\tag{$\ast\ast$}
j_{\bullet}([i])= N_{q-1}^k(v) [i] \qquad \mbox{if } i\in \{2k,2k-1\}.
\end{equation}
To check this last formula, observe that $1-\bar{v}=1-v^{q-1}=(1-v)N_{q-1}(v)$ and that $N_q(v)$ is fixed by the morphism $\Z[G]\to \Z[G]$ induced by $j\colon G \to G$.

Suppose now that $G$ is given with an action on $\widetilde{\Z}$. For $G=\Z$ and $\Z_q$, we will denote by $C_{\bullet}(G;\widetilde{\Z})$ the $\Z$-chain complex obtained by tensoring over $G$ the resolution above with  
$\widetilde{\Z}$ and by $\mathbf j$ the morphism induced by the inversion. The formulas for the product are the same as before and are, for all  cases (including the infinite cyclic case), 
completely described by (\ref{productformula}).

Note that if $g$ is of odd order then we necessarily have $g\cdot t=t$. This leads us to consider the following atomic cases:

\begin{enumerate}
\item[A1.] $G=\Z=\langle u\rangle$ with trivial action on $\widetilde{\Z}$, $u\cdot t=t$.
\\
$C_{\bullet}(G;\widetilde{\Z})$: $\xymatrix{
\quad \Z[1]  \ar[r]^{0}& \Z[0],  
}\quad {\mathbf j}[0]=[0] \qquad {\mathbf j}[1]= -[1]$

\smallskip

\item[A2.] $G=\Z_{p^a}=\langle v~|~ v^{p^a}=1\rangle$, $p$ prime, with trivial action on $\widetilde{\Z}$, $v\cdot t=t$.
\\
$C_{\bullet}(G;\widetilde{\Z})$: $\xymatrix{
\ar[r]^{0}&\Z[2k]  \ar[r]^{p^a}& \Z[2k-1] \ar[r]^{0} &\cdots \ar[r]^{p^a}&\Z[1]  \ar[r]^{0}& \Z[0]  
}$

\smallskip

 \noindent${\mathbf j}([2k-1])= (p^a-1)^k[2k-1] \qquad {\mathbf j}[2k]=(p^a-1)^k[2k]$

\smallskip

\item[A3.] $G=\Z=\langle u\rangle$ with nontrivial action on $\widetilde{\Z}$, $u\cdot t=-t$.
\\
$C_{\bullet}(G;\widetilde{\Z})$: $\xymatrix{
	\quad \Z[1]  \ar[r]^{2}& \Z[0],  
}\quad {\mathbf j}[0]=[0] \qquad {\mathbf j}[1]= [1]$

\smallskip

\item[A4.] $G=\Z_{2^a}=\langle v~|~ v^{2^a}=1\rangle$ with nontrivial action on $\widetilde{\Z}$, $v\cdot t=-t$.
\\
$C_{\bullet}(G;\widetilde{\Z})$: $\xymatrix{
\ar[r]^{2}&\Z[2k]  \ar[r]^{0}& \Z[2k-1] \ar[r]^{2} &\cdots \ar[r]^{0}&\Z[1]  \ar[r]^{2}& \Z[0]  
}$

\smallskip

 \noindent${\mathbf j}[k]=[k]$
\end{enumerate}

\begin{rem}
A similar approach to the study of $H_*(G;\Z)$ may be found in the recent 
article \cite{Hanke}, including  the determination of chain maps induced by homomorphisms $\Z_{p^ a}\to \Z_{p^ b}$ \cite[Proposition 6.4]{Hanke}. In the case of the inversion $\Z_{p^ a}\to \Z_{p^ a}, v\mapsto \bar{v}=v^{p^a-1}$, the underlying chain map at the level of the $\Z[\Z_{p^a}]$-resolution used in \cite{Hanke} differs slightly from our chain map $j_{\bullet}$ described in (\ref{productj}). Nevertheless the two maps are homotopic by virtue of \cite[Lemma 7.4]{Brown}.
\end{rem}

Given a product of infinite and primary cyclic groups $G=G_1\times \cdots \times G_l$, we will then consider the Pontryagin tensor chain algebra 
\[
C_{\bullet}(G;\widetilde{\Z})=C_{\bullet}(G_1;\widetilde{\Z})\otimes \cdots \otimes C_{\bullet}(G_l;\widetilde{\Z})
\] 
provided by these atomic cases. The morphism induced by inversion is given by
\[{\mathbf j}([i_1]\otimes \cdots \otimes [i_l])={\mathbf j}[i_1]\otimes \cdots \otimes {\mathbf j}[i_l]\]

We will often write $[i_1 ~ i_2 ~\cdots ~ i_l]$ (with the convention that the element is $0$ if $i_k<0$ for some $k$) instead of $[i_1]\otimes \cdots \otimes [i_l]$ for the basis elements of the tensor product and refer to these elements as monomials. 

The product on the tensor chain algebra $C_{\bullet}(G;\widetilde{\Z})$, inductively given by 
\[(a\otimes b)\wedge (a'\otimes b')=(-1)^{|a'||b|}(a \wedge a') \otimes (b \wedge b')\] 
is automatically anticommutative. We next state further properties of the chain and homology Pontryagin product which are important for our calculations. These properties could be deduced from a more general structure of a (chain) divided power algebra. Since we are here considering twisted coefficients, we include a short proof. See Remark \ref{rmk:divided power alg} for additional discussion.

\begin{prop} \label{prop:product properties} Let $G=G_1\times \cdots \times G_l$ be a product of infinite and primary cyclic groups. Then the anticommutative product $\wedge$ of the Pontryagin tensor chain algebra $C_{\bullet}(G;\widetilde{\Z})=C_{\bullet}(G_1;\widetilde{\Z})\otimes \cdots C_{\bullet}(G_l;\widetilde{\Z})$ (whose differential is denoted by $\partial$) has the following properties:
	\begin{probs}
		\item for any element $c$ of odd degree, $c\wedge c=0$ (i.e., $\wedge$ is strictly anticommutative);
		\item for any element $c$ of even degree, there exists $e$ such that $c\wedge c=2e$;
		\item for any element $c$ of odd degree, there exists $e$ such that $\partial c\wedge \partial c=2\partial e$.
			\end{probs}
	\end{prop}
\begin{proof} For $c=[i_1 ~ i_2 ~\cdots ~ i_l]$ a monomial, (1) and (2) follow from the formulas (\ref{productformula}). If $c=\sum \lambda_i c_i$ is a linear combination of such monomials, then anticommutativity yields $c\wedge c=\sum \lambda_i^2c_i\wedge c_i$ if $|c|$ is odd, and  $c\wedge c=\sum \lambda_i^2c_i\wedge c_i+ 2\sum\limits_{i<j}\lambda_i\lambda_jc_i\wedge c_j$ if $|c|$ is even. In both cases we can conclude that (1) and (2) hold. To prove (3), we first note that, for $c=\sum \lambda_i c_i$ of odd degree, $\partial c\wedge \partial c=\sum \lambda_i^2\partial c_i\wedge \partial c_i+ 2\sum\limits_{i<j}\lambda_i\lambda_j\partial (c_i\wedge \partial c_j)$. It is hence sufficient to establish (3) for a monomial $c$ of odd degree. We prove by induction on the length that there exists $e$ such that $c\wedge \partial c =2 e$. For $l=1$, this is clearly true by inspection of cases. Let $c=\sigma\otimes [i_l]$ where $\sigma=[i_1 ~ i_2 ~\cdots ~ i_{l-1}]$. Suppose first that $|\sigma|$ is odd. Then $i_l=2k$ with $k\geq 0$. We have $c\wedge \partial c=(\sigma\wedge \partial \sigma)\otimes ([2k]\wedge [2k])$ and the assertion follows by induction. If $|\sigma|$ is even then $i_l=2k+1$. In this case we have $c\wedge \partial c=(\sigma\wedge \sigma)\otimes ([2k+1]\wedge \partial [2k+1])$ and the assertion follows because either $\partial[2k+1]=0$ or $\partial[2k+1]=2[2k]$.
\end{proof}

In homology we have:

\begin{cor}\label{cor:product properties}
	The Pontryagin product of $H_*(G;\widetilde{\Z})$ is strictly anticommutative and has the property that the square of an element of positive even degree is divisible by $2$. 
\end{cor}

\begin{proof}
	It is clear that the strict anticommutativity passes to homology. If $c$ is a cycle of positive even degree then $c\wedge c$ is also a cycle and, since $C_{\bullet}(G;\widetilde{\Z})$ is $\Z$-free we can conclude that the element $e$ of Proposition \ref{prop:product properties}(2) is also a cycle. This together with Proposition \ref{prop:product properties}(3) permits us to see that $[c]\wedge [c]$ is divisible by $2$. 
\end{proof}

\begin{rem}\label{rmk:divided power alg}
	As mentioned before, we could more generally state that $C_{\bullet}(G;\widetilde{\Z})$ is a divided power chain algebra as defined in \cite{Richter}. More intrinsically, the bar resolution $\mathrm{Bar}_{\bullet}(G)$ admits the structure of divided power chain algebra and we could follow the 
lines of \cite{Brown} to prove Proposition \ref{prop:product properties} and Corollary \ref{cor:product properties}.
	Note that the atomic chain algebras can be described as the following differential exterior/divided power algebras given with a morphism $\mathbf{j}$ induced by the inversion. 	
The case (A1) corresponds to a differential exterior algebra $(\bigwedge x,d)$ where $|x|=1$, $dx=0$ and the inversion is given by $\mathbf{j}(x)=-x$. The case (A3) is similar with $dx=2$ and $\mathbf{j}=\mathrm{Id}$. Case (A2) corresponds to a differential divided power algebra $(\bigwedge x \otimes \Gamma y,d)$ where $|x|=1$, $|y|=2$, $dx=0$, $dy^{(i)}=p^axy^{(i-1)}$ (here $y^{(i)}$ denotes the divided power $\frac{y^i}{i!}$) and the inversion is given by $\mathbf{j}(x)= (p^a-1)x$ and $\mathbf{j}(y^{(i)})= (p^a-1)y^{(i)}$.  The case (A4) is similar with $dx=2$, $dy=0$ and $\mathbf{j}=\mathrm{Id}$.
\end{rem}

\subsection{Proof of Proposition \ref{prop-general}}\label{subsection:proof-prop-general}
We divide the general statement of Proposition \ref{prop-general} in several propositions. In all these propositions, we suppose $n\geq 2$ and consider a class ${\mathfrak c}\in H_n(G;\widetilde{\Z})$. In the proofs, we will often use the letter $c$ to denote a cycle representing the class $\mathfrak{c}$. As before, the images of these elements under the morphisms induced by the inversion will be denoted by $\mathbf{j}(c)$ and $j_*(\mathfrak{c})$.

\begin{prop} \label{twistedaction} Suppose that $G\cong A\times G'$ where $G'$ is a finitely generated abelian group and $A=\Z$ or $\Z_ {2^a}$. If $A$ acts nontrivially on $\widetilde{\Z}$, then $\chi_*(\mathfrak{c}\times \mathfrak{c})=0$.
\end{prop}

\begin{proof}
Since $A$ acts nontrivially on $\widetilde{\Z}$, the homology group $H_i(A;\widetilde{\Z})$ is, for any $i\geq 0$, either $0$ or $\Z_2$. The K\"unneth exact sequence permits us to see that, for any finitely generated abelian group $B$ and for any $i\geq 0$, the homology group $H_i(A\times B;\widetilde{\Z})$ is a (possibly null) direct sum of copies of $\Z_2$. Writing $G'=G'_0\times G'_2\times G'_{odd}$ where $G'_0$ is free, $G'_2$ is a product of cyclic groups of the form $\Z_{2^a}$ and $G'_{odd}$ is a product of cyclic groups of odd order, we can also see that the inclusion $A\times G'_0 \times G'_2 \to A\times G'$ induces an isomorphism  $H_i(A\times G'_0\times G'_2;\widetilde{\Z})\cong  
H_i(A\times G';\widetilde{\Z})$ for each $i$. Since this inclusion is a homomorphism, we actually have an isomorphism of algebras $H_*(A\times G'_0\times G'_2;\widetilde{\Z})\cong  H_*(A\times G';\widetilde{\Z})$ which is compatible with the morphism $j_*$ induced by the inversion. 

In order to establish the statement, it is therefore sufficient to consider classes $ \mathfrak{c}\in H_{n}(A\times G'_0\times G'_2;\widetilde{\Z})$. Let $\mathfrak{c}$ be such a class and let $C_{\bullet}$ be the $\Z$-chain complex associated to a decomposition of $A\times G'_0\times G'_2$ as a product of atomic cyclic groups. The class $\mathfrak{c}$ can be represented by a cycle $c=\sum \lambda_i \sigma_i$ where $\lambda_i\in \Z$ and the $\sigma_i$ are the basis monomials of $C_n$. As noted before, $H(C_{\bullet})=H_*( A\times G'_0\times G'_2;\widetilde{\Z})$ is, in each degree, a (possibly null) direct sum of copies of $\Z_2$. Therefore $H(C_{\bullet})=H(C_{\bullet})\otimes \Z_2$ and, since $H(C_{\bullet})\otimes \Z_2 \to H(C_{\bullet}\otimes \Z_2)$ is injective, we can work with the cycle $\tilde{c}=\sum  \sigma_i\in C_n\otimes \Z_2$. From the expression of the inversion in the atomic cases, we see that ${\mathbf j}(\sigma_i)=\sigma_i$ for each $i$ and therefore $j_*(\mathfrak{c})=\mathfrak{c}$. We then obtain $\chi_*(\mathfrak{c}\times \mathfrak{c})=\mathfrak{c} \wedge j_*(\mathfrak{c})=\mathfrak{c} \wedge \mathfrak{c}$. Since a Pontryagin square is always divisible by $2$ and $H(C_{\bullet})$ is of $2$-torsion, we conclude that $\chi_*(\mathfrak{c}\times \mathfrak{c})=0$.
\end{proof}

From now on we will consider only the trivial action on $\widetilde{\Z}$ and suppress the coefficients in homology as well as in chain complexes. 

\begin{prop} \label{torus} Suppose that $G=\Z^ r$ acts trivially on $\widetilde{\Z}$ and that either $n$ is odd, or $n$ is even such that $2n>r$. Then we have $\chi_*(\mathfrak{c}\times \mathfrak{c})=0$. 
\end{prop}

\begin{proof} Here we can work at the homology level since $H_*(\Z^r)=(H_*(\Z))^{\otimes r}$. This algebra is the $\Z$-exterior algebra $\bigwedge(x_1,\cdots,x_r)$ where $|x_1|=\cdots=|x_r|=1$ and we have $j_*(x_i)=-x_i$ for any $i$. Consequently, for any monomial $\sigma$, we have $j_*(\sigma)=(-1)^{|\sigma|}\sigma$. When $n$ is odd, we obtain $j_*(\mathfrak{c})=-\mathfrak{c}$ and $\chi_*(\mathfrak{c}\times \mathfrak{c})= -\mathfrak{c} \wedge \mathfrak{c}=0$. When $n$ is even, we have $j_*(\mathfrak{c})=\mathfrak{c}$ and $\chi_*(\mathfrak{c}\times \mathfrak{c})=\mathfrak{c} \wedge \mathfrak{c}$. In this case, the result follows from the fact that, if $\sigma$ and $\sigma'$ are two monomials of degree 
$n>r/2$, then they have some $x_i$ in common and $\sigma\wedge \sigma'=0$. 
\end{proof}

\begin{ex}\label{ex:cond_a}
The hypothesis $2n>r$ in Proposition \ref{torus} is sharp. For instance, for $r=4$ and $n=2$, the homology class $\mathfrak{c}=x_1x_2+x_3x_4$
satisfies  $j_*(\mathfrak{c})=\mathfrak{c}$ and $\chi_*(\mathfrak{c})=\mathfrak{c}\wedge \mathfrak{c}=2x_1x_2x_3x_4\neq 0$.
\end{ex}

\begin{prop} \label{torusxcyclic} Suppose that $G=\Z^ r\times \Z_q$ ($r\geq 0$, $q\geq 2$) acts trivially on $\widetilde{\Z}$. If $n>r$ or $n=r=2m$ (with $m\geq 0$) then $\chi_*(\mathfrak{c}\times \mathfrak{c})=0$. 
\end{prop}

\begin{proof} By the K\"unneth formula, we see that $H_*(G)=H_*(\Z^r)\otimes H_*(\Z_q)$. Suppose first that $n>r$. We can write $\mathfrak{c}=\sum \sigma_i\otimes \alpha_i$ where $\sigma_i\in H_*(\Z^ r)$ and $\alpha_i\in H_{>0}(\Z_q)$. From the expression of $j_{\bullet}$ for $\Z_q$ given in \ref{productj}), we then have $j_*(\mathfrak{c})=\sum \lambda _i \sigma_i\otimes \alpha_i$ where $\lambda_i\in \Z$. Since $H_{>0}(\Z_q)$ is 
concentrated in odd degrees and $\alpha_i\wedge \alpha _j=0$ for any $i,j$ we can conclude that $\chi_*(\mathfrak{c}\times \mathfrak{c})=\mathfrak{c} \wedge j_*(\mathfrak{c})=0$. If $n=r$, we have $\mathfrak{c}=\mathfrak{c}'+\sum \sigma_i\otimes \alpha_i$ where $\mathfrak{c}'\in H_r(\Z^r)=\Z$ is of maximal length in the exterior algebra $H_*(\Z^r)$ and $\sum \sigma_i\otimes \alpha_i$ is as before. Since $n=r$ is even, we then obtain  $j_*(\mathfrak{c})=\mathfrak{c}'+\sum \lambda _i \sigma_i\otimes \alpha_i$ and $\mathfrak{c} \wedge j_*(\mathfrak{c})=\mathfrak{c}'\wedge \mathfrak{c}'+2\sum \lambda _i (\mathfrak{c}'\wedge\sigma_i)\otimes \alpha_i$, which is $0$ because $|\sigma_i|$ must be positive and  $\mathfrak{c}'$ is of maximal length.
\end{proof}

When $n$ is odd, the hypothesis $n>r$ is sharp as we can see in the following example.
\begin{ex}\label{ex:cond_b}
Let $n=r=3$ and also let $q=3$. We work at the chain level, in $C_*(\Z)^{\otimes 3}\otimes C_*(\Z_3)$, and write $[i_1~i_2~i_3 ~ i_4]$ for the basis elements. Consider the cycle $c=[1110]+[0003]$. We have ${\mathbf j}(c)=-[1110]+[0003]$ and 
\[ c\wedge {\mathbf j}(c)=2[1113]. 
\]
Therefore $\mathfrak{c} \wedge j_*(\mathfrak{c})=2[1113] \in H_{6}(G)=\oplus \Z_3$ and is not $0$. Observe that $j_*(\mathfrak{c})\neq \pm \mathfrak{c}$. Also note that, here and in examples below, we write 
$\oplus \Z_p$ to briefly indicate a direct sum of copies of $\Z_p$. 
\end{ex}

The following example shows that the hypothesis $n\geq r$ is sharp when $n$ is even.
\begin{ex}\label{ex:cond_b_even}
Let $n=6$, $r=7$ and $q=3$. Consider the cycle $c=[11111100]+[00000015]$. We have ${\mathbf j}(c)=c$ and 
\[ c\wedge {\mathbf j}(c)=2[11111115]. 
\]
Therefore $\mathfrak{c} \wedge j_*(\mathfrak{c})\in H_{12}(G)=\oplus \Z_3$ does not vanish.		
\end{ex}

\begin{prop} \label{Zp-product} Suppose that $G$ is of one of the following forms (where $p$ is a prime and $a$, $b$, $c$ are positive integers) and acts trivially on $\widetilde{\Z}$.
\begin{itemize}
\item[(i)] $\Z^r \times \Z_{p^ a}\times \Z_{p^ b}$ with $r\in \{0,1\}$.
\item[(ii)] $\Z_{p^ a}\times \Z_{p^ b}\times \Z_{p^c}$. 
\end{itemize}
Then $\chi_*(\mathfrak{c}\times \mathfrak{c})=0$.
\end{prop}

\begin{proof} Recall that, for $e\geq 1$, the chain complex $C_{\bullet}(\Z_{p^e})$ is given by 
\[\xymatrix{
\ar[r]^{0}&\Z[2k]  \ar[r]^{p^e}& \Z[2k-1] \ar[r]^{0} &\cdots \ar[r]^{p^e}&\Z[1]  \ar[r]^{0}& \Z[0]  
}\]
and that we have (in particular) $[2k-1]\wedge [2l-1]=0$, ${\mathbf j}([2k-1])= (p^e-1)^k[2k-1]$, and ${\mathbf j}([2k])=(p^e-1)^k[2k]$, see \eqref{productformula} and the surrounding discussion. 

\smallskip

\noindent\emph{Case} (i). We first suppose that $r=0$. Let $G=\Z_{p^ a}\times \Z_{p^ b}$. We can suppose that $a\leq b$. From the K\"unneth formula, we can see that the Tor terms will only give odd dimensional classes. Therefore, if $n=2m>0$ is positive and even, we have
\[H_{2m}(G)= \bigoplus\limits_{i+j=2m}(H_i(\Z_{p^a})\otimes H_j(\Z_{p^b}))=\bigoplus\limits_{k+l=m}(H_{2k-1}(\Z_{p^a})\otimes H_{2l+1}(\Z_{p^b})).
\]
Consequently, $H_{2m}(G) \wedge H_{2l}(G)=0$ for any $l>0$.
We conclude that $c\wedge {\mathbf j}(c)=0$.

If $n=2m-1$ is odd, then 
\[H_{2m-1}(G)=H_{2m-1}(\Z_{p^a})\otimes \Z \oplus \Z \otimes H_{2m-1}(\Z_{p^b}) \oplus 
\bigoplus\limits_{k+l=m}{\rm Tor}(H_{2k-1}(\Z_{p^a}),H_{2l-1}(\Z_{p^b})).
\]
We will see that $j_*(\mathfrak{c})=(-1)^m\mathfrak{c}$. If $\mathfrak{c}\in H_{2m-1}(\Z_{p^a})\otimes \Z$, then $\mathfrak{c}$ is represented in the associated chain complex by the cycle $c=\zeta [2m-1 ~0]$ ($\zeta \in \Z$) and we have ${\mathbf j}(c)=(p^a-1)^m c$. Since $H_{2m-1}(\Z_{p^a})\otimes \Z =\Z_{p^a}$, we have $j_*(\mathfrak{c})=(-1)^m\mathfrak{c}$. Analogously, if $\mathfrak{c}\in \Z\otimes H_{2m-1}(\Z_{p^b})=\Z_{p^b}$, then $\mathfrak{c}$ is represented by a multiple of $[0 ~2m-1]$ and  $j_*(\mathfrak{c})=(-1)^m\mathfrak{c}$. If $\mathfrak{c}\in {\rm Tor}(H_{2k-1}(\Z_{p^a}),H_{2l-1}(\Z_{p^b}))$ with $k+l=m$ observe that, since $a\leq b$, we have ${\rm Tor}(H_{2k-1}(\Z_{p^a}),H_{2l-1}(\Z_{p^b}))=\Z_{p^a}$. The class $\mathfrak{c}$ is represented by (a multiple of) the cycle
\[c=[2k-1 ~ 2l]+p^{b-a}[2k ~2l-1]\]
and we thus have ${\mathbf j}(c)=(p^a-1)^k(p^b-1)^lc$. Passing to homology we obtain $j_*(\mathfrak{c})=(-1)^{k+l}\mathfrak{c}=(-1)^m\mathfrak{c}$. Finally, by linearity, we can conclude that, for any $\mathfrak{c}\in H_{2m-1}(G)$, we have $j_*(\mathfrak{c})=(-1)^m\mathfrak{c}$. Therefore $\chi_*(\mathfrak{c}\times \mathfrak{c})=\mathfrak{c} \wedge j_*(\mathfrak{c})=(-1)^m\mathfrak{c} \wedge\mathfrak{c}=0$ since the degree of $\mathfrak{c}$ is odd and $H_*(G)$ is a Pontryagin algebra.

From this analysis, we can deduce that, for $G=\Z_{p^a}\times \Z_{p^b}$ and nonzero even degrees,
\[ H_{even}(G)\wedge H_{even}(G)=0 \quad \mbox{and} \quad  H_{even}(G)\wedge H_{odd}(G)=0.\]

Suppose now that $r=1$. We then have 
\[H_n(G)=H_0(\Z) \otimes H_n(\Z_{p^a}\times \Z_{p^b}) \oplus  H_1(\Z) \otimes H_{n-1}(\Z_{p^a}\times \Z_{p^b})\]
and, using the previous calculations together with the fact that $n\geq 2$, we can check that any $\mathfrak{c} \in H_{n}(G)$ satisfies $\mathfrak{c} \wedge j_*(\mathfrak{c})=0$.

\smallskip

\noindent\emph{Case} (ii). Let $G=A\times B$ where $A=\Z_{p^ a}$, $B=\Z_{p^ b}\times \Z_{p^c}$ and $a\leq b \leq c$. By the K\"unneth formula, in even dimension $n=2q>0$, we have
\[H_{2q}(G)= \Z\otimes H_{even}(B) \oplus H_{odd}(A)\otimes H_{odd}(B) \oplus {\rm Tor}(H_{odd}(A), H_{even}(B)).
\]
Check that the generators of  $H_{odd}(A)\otimes  {\rm Tor}(H_{*}(\Z_{p^ b}), H_{*}(\Z_{p^c}))\subset H_{odd}(A)\otimes H_{odd}(B)$ and of $ {\rm Tor}(H_{odd}(A), H_{even}(B))$ are of 
the forms 
\[
\begin{aligned}
&[2m-1 ~ 2k~2l-1] + p^{b-a} [2m ~2k-1 ~2l-1], \\ 
&[2m-1 ~ 2k-1~2l] - p^{c-a} [2m ~2k-1 ~2l-1], \ \text{and}\\ 
&[2m-1 ~ 2k-1~2l] + p^{c-b} [2m-1 ~2k ~2l-1].
\end{aligned}
\] 
Using this, any homology class in $H_{2q}(G)$ can be represented by a linear combination of monomials of the form
\[\sigma_L=[l_1 ~ l_2 ~ l_3]\]
where $l_1+l_2+l_3=2q$ and two among $l_1,l_2,l_3$ are odd. As we have, for such monomials, $\sigma_L\wedge \sigma_{L'}=0$, we can conclude that $H_{2q}(G)\wedge H_{2q}(G)=0$ and that for any $\mathfrak{c} \in H_{2q}(G)$, $\chi_*(\mathfrak{c}\times \mathfrak{c})=\mathfrak{c} \wedge j_*(\mathfrak{c})=0$.

Now, in odd dimension $n=2q-1$, we have
\[H_{2q-1}(G)= \Z\otimes H_{odd}(B) \oplus H_{odd}(A)\otimes H_{even}(B) \oplus {\rm Tor}(H_{odd}(A), H_{odd}(B)),
\]
and we can check that the generators of $ {\rm Tor}(H_{odd}(A), H_{odd}(B))$ are of the form 
\[[2m-1 ~ 2k~2l] + p^{b-a} [2m ~2k-1 ~2l] + p^{c-a} [2m ~2k ~2l-1]\]
{Recall the convention that} $[l_1 ~ l_2 ~ l_3]=0$ if $l_i<0$ for some $i$.
Considering the above decomposition of $H_{2q-1}(G)$, let 
\[
\mathfrak{c} \in \Z\otimes H_{odd}(B) \oplus H_{odd}(A)\otimes H_0(B) \oplus {\rm Tor}(H_{odd}(A), H_{odd}(B)),
\] 
and let $\mathfrak{d} \in  H_{odd}(A)\otimes H_{even>0}(B)$. 
We can check that $j_*(\mathfrak{c})=(-1)^q\mathfrak{c}$ while $j_*(\mathfrak{d})=-(-1)^{q}\mathfrak{d}$. We can also check that $\mathfrak{d}$ can be represented by a linear combination of monomials of the form
\[\sigma_L=[l_1 ~ l_2 ~ l_3]\]
where $l_1+l_2+l_3=2q-1$ and all $l_1,l_2,l_3$ are odd, while $\mathfrak{c}$ can be represented by such a monomial where at least one $l_i$ is odd. From all these facts (together with the fact that the Pontryagin square of an odd element is $0$) we deduce:
\begin{itemize}
\item $\mathfrak{c} \wedge j_*(\mathfrak{c})=(-1)^q\mathfrak{c} \wedge\mathfrak{c}=0$
\item $\mathfrak{d} \wedge j_*(\mathfrak{d})=-(-1)^q\mathfrak{d} \wedge\mathfrak{d}=0$
\item  $\mathfrak{c} \wedge\mathfrak{d}=0$
\end{itemize}
and we obtain that $\chi_*((\mathfrak{c}+\mathfrak{d})\times (\mathfrak{c}+\mathfrak{d}))=(\mathfrak{c}+\mathfrak{d}) \wedge (j_*(\mathfrak{c})+j_*(\mathfrak{d}))=0$.
\end{proof}

Independent of dimension, the following examples show that, in Proposition \ref{Zp-product}, we can not relax the hypothesis $r\leq 1$ in the first case, that we can not expect to have a $\Z$ factor in the second case, and that we can not expect to have a ``third case'' with four factors. 

\begin{ex}\label{ex:cond_c}
Let $G=\Z^2\times \Z_3\times \Z_3$. Consider the cycle $c=[1030]+[0103]$ in $C_*(\Z)^{\otimes 2}\otimes C_*(\Z_3)\otimes C_*(\Z_3)$. We have $\mathfrak{c} \wedge j_*(\mathfrak{c})=2[1133]$ which is not $0$ in $H_{8}(G)=\oplus \Z_3$.		
\end{ex}

\begin{ex}\label{ex:cond_c_odd}
{Let $G=\Z^2\times \Z_3\times \Z_3$}. Consider the cycle $c=[1130]+[0005]$ in $C_*(\Z)^{\otimes 2}\otimes C_*(\Z_3)\otimes C_*(\Z_3)$. We have $\mathfrak{c} \wedge j_*(\mathfrak{c})=[1135]$ which is not $0$ in $H_{10}(G)=\oplus \Z_3$.
\end{ex}

\begin{ex}\label{ex:cond_d_1}
Let $G=\Z\times (\Z_3)^3$. Consider the cycle $c=[1330]+[0007]$ in $C_*(\Z)\otimes C_*(\Z_3)^{\otimes 3}$. We have $\mathfrak{c} \wedge j_*(\mathfrak{c})=2[1337]$ which is not $0$ in $H_{14}(G)=\oplus \Z_3$.		
\end{ex}

\begin{ex}\label{ex:cond_d_1_even}
{Let $G=\Z\times (\Z_3)^3$.} Consider the cycle $c=[1003]+[0121]+[0112]$ in $C_*(\Z)\otimes C_*(\Z_3)^{\otimes 3}$. We have $\mathfrak{c} \wedge j_*(\mathfrak{c})=[1115]$ which is not $0$ in $H_{8}(G)=\oplus \Z_3$.		
\end{ex}

\begin{ex}\label{ex:cond_d_2}
Let $G=(\Z_3)^4$. Consider the cycle $c=[3300]+[0033]$ in $C_*(\Z_3)^{\otimes 4}$. We have $\mathfrak{c} \wedge j_*(\mathfrak{c})=2[3333]$ which is not $0$ in $H_{12}(G)=\oplus \Z_3$.		
\end{ex}

\begin{ex}\label{ex:cond_d_2odd}
{Let $G=(\Z_3)^4$.} Consider the cycle $c=[0311]+[4100]+[3200]$ in 
$C_*(\Z_3)^{\otimes 4}$. We have  $\mathfrak{c} \wedge j_*(\mathfrak{c})=[3511]$ which is not $0$ in $H_{10}(G)=\oplus \Z_3$.
\end{ex}

We finally prove Proposition \ref{prop-general}(2e) through the following 
proposition.

\begin{prop} Suppose that $G=\Z^ r\times (\Z_2)^s$ 
acts trivially on $\widetilde{\Z}$ and that $n$ is either odd, or $n$ is even such that $2n>r$. Then $\chi_*(\mathfrak{c}\times \mathfrak{c})=0$.
\end{prop}

\begin{proof} We have $H_n(G)=H_n(\Z^r)\oplus \bigoplus\limits_{k+l=n, l>0} H_k(\Z^r)\otimes H_l((\Z_2)^s)$ and, for $l>0$, $H_l((\Z_2)^s)=\oplus\Z_2$. We can write $\mathfrak{c}=\mathfrak{c}'+\sum \sigma_i\otimes \alpha_i$ where $\mathfrak{c}',\sigma_i\in H_*(\Z^ r)$ and $\alpha_i\in H_{>0}((\Z_2)^s)$. It follows from Corollary \ref{cor:product properties} that $\alpha_i^2=0$ for any $i$. We then have $j_*(\mathfrak{c})=(-1)^n\mathfrak{c}'+\sum \sigma_i\otimes \alpha_i$ and $\chi_*(\mathfrak{c}\times \mathfrak{c})=\mathfrak{c} \wedge j_*(\mathfrak{c})=(-1)^n \mathfrak{c}'\wedge 
\mathfrak{c}'$, which is $0$ as in Proposition \ref{torus}.
\end{proof}

The hypothesis $2n>r$, when $n$ is even, is sharp as we can see in the following somewhat trivial example
\begin{ex}\label{ex:cond_e}
Let $n=4$ and $r=8$. Consider the cycle $c=[111100000]+[000011110]$ in $C_*(\Z)^{\otimes 8}\otimes C_*(\Z_2)$. We have \[\mathfrak{c} \wedge j_*(\mathfrak{c})=2[111111110] \in \Z=H_{8}(\Z^8)\otimes \Z\subset 		H_{8}(\Z^8\times \Z_2)\] which is not $0$.	
\end{ex}

\section{Examples} \label{sec:examples}
The goal of this section is to construct examples which show that our conditions in Theorem \ref{thm:TC} and \ref{thm:cofibre} are sharp. Most of them are obtained by surgery from connected sums of orientable manifolds.

Let  $M_1$, $M_2$ be $n$-dimensional orientable manifolds. For the sake of simplicity, we suppose that $n\geq 4$. As in \cite{MilnorSurgery}, we will consider the connected sum $M=M_1\# M_2$ constructed in order to have an orientation compatible with the orientations of $M_1$ and $M_2$. In 
particular the two projections $M_1\# M_2 \to M_1$ and 
$M_1\# M_2 \to M_2$ have degree $1$ and the pinch map $M_1\# M_2 \to M_1\vee M_2$ maps $[M_1\# M_2]$ to the sum $[M_1]+[M_2] \in H_n(M_1) \oplus H_n(M_2)$. Note that the map $M_1\# M_2 \to M_1\vee M_2$ induces an isomorphism of fundamental groups. 

For $i\in\{1,2\}$, let $G_i=H_1(M_i)$ and let $\alpha_*({\mathfrak m}_i)\in H_n(G_i)$ be the image of $[M_i]$ through the morphism induced in homology by $M_i\xrightarrow{\gamma_i} B\pi_1(M_i) \xrightarrow{\alpha} BG_i$ where, as in Section 2, $\pi_1(\gamma_i)$ is an isomorphism and $\alpha$ is the abelianization. The map $M_1\# M_2 \to M_1\vee M_2\xrightarrow{\gamma_1\vee \gamma_2} B\pi_1(M_1)\vee B\pi_1(M_2)$ induces an isomorphism of fundamental groups and we denote by ${\mathfrak m} \in H_n(\pi_1(M))=H_n(\pi_1(M_1)*\pi_1(M_2))$ the image of $[M]=[M_1\# M_2]$ through the morphism induced by this map in homology. Since $H_1(M)$ is isomorphic 
to $G=G_1\times G_2$ and the inclusion $BG_1 \vee BG_2 \hookrightarrow BG_1 \times BG_2$ induces an abelianization homomorphism of fundamental groups, the class $\alpha_*({\mathfrak m}) \in H_n(G)$ considered in Proposition \ref{prop:cofibre} can be identified with $f_*([M])$ where $f$ is the map
\[
M=M_1\# M_2\xrightarrow{pinch} M_1\vee M_2 \xrightarrow{(\alpha\gamma_1)\vee (\alpha\gamma_2)} BG_1\vee BG_2 \hookrightarrow BG_1\times BG_2=BG.
\]
We then have 
\[\alpha_*({\mathfrak m})=f_*([M])=\alpha_*({\mathfrak m}_1) \times 1 
+ 1\times \alpha_*({\mathfrak m}_2)\in H_n(G_1\times G_2)=H_n(G).\] 

From Propositions \ref{prop:cofibre} and \ref{prop:TCcofibre}, we know that $f_*([M])\wedge j_*(f_*([M]))\neq 0$ implies $\cat(C_{\Delta}(M))=\TC(M)=2n$. 
Note that, by construction, the homomorphism $\pi_1(f):\pi_1(M)\to G=H_1(M)$ is surjective but is not an isomorphism.

In the examples below we start with a connected sum $M=M_1\# M_2$ and a 
map $f:M\to BG$ as above such that $f_*([M])\wedge j_*(f_*([M]))\neq 0$. In order to obtain a limiting example for Theorems \ref{thm:TC} and \ref{thm:cofibre}, we use (iterated) surgery for killing the kernel of $\pi_1(f):\pi_1(M) \to G$ and producing an orientable $n$-manifold $N$ together with a map $g:N\to BG$ which induces an isomorphism of fundamental groups. The manifold $N$ so obtained has hence an abelian fundamental group and the map $g$ satisfies $g_*([N])=f_*([M])$ in $H_n(G;\Z)$ (see \cite{MilnorSurgery} and also \cite[proof of Theorem 5.2]{Dranishnikov-EssMfld} 
and \cite{Kutsak} where surgery techniques are applied to a similar situation). Since $g_*([N])=f_*([M])$ in $H_n(G;\Z)$, we have $g_*([N])\wedge j_*(g_*([N]))\neq 0$ and we can then conclude (by Proposition \ref{prop:equivalence}) that $\TC(N)$ and $\cat(C_{\Delta}(N))$ are both maximal.

In each example, the class $f_*([M])$ corresponds essentially to one of the algebraic examples of Section 3. In some cases we make a slight modification to satisfy the requirement $n\geq 4$ coming from the surgery procedure. 

\begin{ex} \label{ex:intro} \emph{A manifold $N$ with $\pi_1(N)=\Z^8$, $\dim(N)=4$ and $\TC(N)=2\dim(N)$.}
Let $T^k=(S^1)^k$. We start with $M=T^4 \# T^4$ and the map $f$
\[\xymatrix{
M=T^4\# T^4 \ar[r]^-{pinch} & T^4\vee T^4 \ar@{^(->}[r] & T^4 \times T^4=B(\Z^4 \times \Z^4)		
}\]
which induces the abelianization homomorphism $\pi_1(f): \pi_1(M)=\Z^4 * \Z^4 \to G=\Z^4 \times \Z^4$.
The kernel of $\pi_1(f)$ is, as a normal subgroup, generated by a finite number of commutators. By killing these commutators we obtain an orientable manifold $N$ of dimension $4$ and a map $g:N\to BG$ which induces an 
isomorphism of fundamental groups and such that $g_*([N])=f_*([M])$ in $H_4(\Z^4 \times \Z^4)$. With the notations of Section 3, this element can be identified with the class $\mathfrak{c}=[11110000]+[00001111]$ and, as in Example \ref{ex:cond_a}, $\mathfrak{c}\wedge j_*(\mathfrak{c})=2[11111111]\neq 0$ in $H_{8}(\Z^4\times \Z^4)=\Z$.
We then obtain $\TC(N)=\cat(C_{\Delta}(N))=8$. 

This example shows that, when $n$ is even, the condition $r<2n$ in Theorem \ref{thm:TC}(2a) and Theorem \ref{thm:cofibre}(2a) is sharp. 
\end{ex}

\begin{ex} \emph{A manifold $N$ with $\pi_1(N)=\Z^7\times \Z_3$, $\dim(N)=7$ and $\TC(N)=2\dim(N)$.}
We start with $M=T^7 \# L_3^7$ where $L_3^7$ is the lens space $S^7/\Z_3$ and we consider the map $f$
\[\xymatrix{
M=T^7\# L_3^7 \ar[r]^-{pinch} & T^7\vee L_3^7 \ar@{^(->}[r] & T^7 \times L_3^{\infty}=B(\Z^7 \times \Z_3)
}\]
which induces the abelianization homomorphism $\pi_1(f): \pi_1(M)=\Z^7 * \Z_3 \to G=\Z^7 \times \Z_3$. After killing $\ker\pi_1(f)$, we obtain 
an orientable manifold $N$ of dimension $7$ and a map $g:N\to BG$ which induces an isomorphism of fundamental groups and such that ${\mathfrak 
c}=g_*([N])=f_*([M])=[11111110]+[00000007]$. As in Example \ref{ex:cond_b} we have ${\mathfrak c}\wedge j_*{\mathfrak c}=2[11111117]\neq 0$ and conclude that $\TC(N)=\cat(C_{\Delta}(N)=14$.
		 
This example shows that the condition $r<n$ in Theorem \ref{thm:TC}(2b) 
and Theorem \ref{thm:cofibre}(2b) is sharp when $n$ is odd. This is also an example of an odd dimensional manifold where we have $j_*g_*([N])\neq \pm g_*([N])$.	
\end{ex}

We now briefly summarize several further examples.  The first example shows that the condition $n\geq r$ in Theorem \ref{thm:TC}(2b) and Theorem \ref{thm:cofibre}(2b) is sharp {when $n$ is even}.

\begin{ex} 
Starting with $M=T^6 \# (S^1 \times L_3^5)$ which realizes the class $[11111100]+[00000015]\in H_6(\Z^7 \times \Z_3)$ of Example \ref{ex:cond_b_even}, we obtain a manifold $N$ with $\pi_1(N)=\Z^7\times \Z_3$, $\dim(N)=6$ and $\TC(N)=\cat(C_{\Delta}(N))=2\dim(N)$. 
\end{ex}

The two following examples show that the condition $r<2$ in Theorem \ref{thm:TC}(2c) and Theorem \ref{thm:cofibre}(2c) is sharp.

\begin{ex} 
Starting with $M=(S^1\times L_3^3) \# (S^1 \times L_3^3)$ which realizes the class $[1030]+[0103]\in H_4(\Z^2 \times \Z_3 \times \Z_3)$ of Example \ref{ex:cond_c}, we obtain a manifold $N$ with $\pi_1(N)=\Z^2\times \Z_3\times \Z_3$, $\dim(N)=4$ and $\TC(N)=\cat(C_{\Delta}(N))=2\dim(N)$.
\end{ex}

\begin{ex} 
{Starting} with $M=(S^1\times S^1\times L_3^3) \# L_3^5$ which realizes the class $[1130]+[005]\in H_5(\Z^2 \times \Z_3 \times \Z_3)$ of Example \ref{ex:cond_c_odd}, we obtain a manifold $N$ with $\pi_1(N)=\Z^2\times \Z_3\times \Z_3$, $\dim(N)=5$ and $\TC(N)=\cat(C_{\Delta}(N))=2\dim(N)$. 
\end{ex}

The following example shows that, when $n$ is odd, we can not expect a free part in Theorems \ref{thm:TC}(2d) and \ref{thm:cofibre}(2d). See below 
for the case $n$ even.

\begin{ex} 
Starting with $M=(S^1 \times L_3^3\times L_3^3) \#  L_3^7$ which realizes the class $[1330]+[0007]\in {H_7}(\Z \times \Z_3^3)$ of Example \ref{ex:cond_d_1}, we obtain a manifold $N$ with $\pi_1(N)=\Z\times \Z_3^3$, $\dim(N)=7$ and $\TC(N)=\cat(C_{\Delta}(N))=2\dim(N)$. 	
\end{ex}

The following example shows that, when $n$ is even, we can not expect an extension to a four-fold product of Condition (2d) of Theorems \ref{thm:TC} and \ref{thm:cofibre}. See below for the case $n$ odd.
\begin{ex} 
Starting with $M=(L_3^3\times L_3^3) \# (L_3^3\times L_3^3)$ which realizes the class $[3300]+[0033]\in H_7(\Z_3^4)$ of Example \ref{ex:cond_d_2}, we obtain a manifold $N$ with $\pi_1(N)=\Z_3^4$, $\dim(N)=12$ 
and $\TC(N)=\cat(C_{\Delta}(N))=2\dim(N)$. 	
\end{ex}

The following example shows that, when $n$ is even, the condition $r<2n$ in Theorem \ref{thm:TC}(2e) and Theorem \ref{thm:cofibre}(2e) is sharp. 
\begin{ex} 
Let $P=\R\P^3\times S^3$. The composition of the projection $\R\P^3\times S^3 \to \R\P^{3}$ with the inclusion $\R\P^3 \to \R\P ^{\infty}$ yields a map $P \to \R\P^{\infty}$ which induces an isomorphism of fundamental groups and sends $[P]$ to $0\in H_6(\Z_2)$. Starting with $M=(T^6) \# (T^6\# P)$ we can then realize the class ${\mathfrak c}=[1111110000000]+[0000001111110]\in H_{6}(\Z^{12}\times \Z_2)$, which is similar to the class considered in Example \ref{ex:cond_e}. Checking that $\mathfrak{c}\wedge j_*\mathfrak{c}\neq 0$, we then obtain a manifold $N$ with $\pi_1(N)=\Z^{12}\times \Z_2$, $\dim(N)=6$ and $\TC(N)=\cat(C_{\Delta}(N))=2\dim(N)$. 	
\end{ex}

{Finally}, in order to obtain examples which show that we can not expect a free term in Theorem \ref{thm:TC}(2d) and Theorem \ref{thm:cofibre}(2d) 
{when $n$ is even} and that we can not expect an extension of these two statements to a four-fold product when 
$n$ is odd, we consider the classes $\mathfrak{c}$ of Example \ref{ex:cond_d_1_even} and Example \ref{ex:cond_d_2odd}, which are both of dimension 
$4\leq n \leq 5$. First, since $n\leq 5$, we can invoke Thom's realizability theorem \cite[Theorem II.27]{Thom} to obtain an $n$-dimensional orientable manifold $M$ and a map $f:M\to BG$ such that $f_*([M])=\mathfrak c$. As $\pi_1(f)$ may not be surjective, we can follow the procedure described in the proof of \cite[Theorem 2.1]{Kutsak} to 
obtain $\tilde{f}:\widetilde{M}\to BG$ such that $\pi_1(\tilde{f})$ is surjective and $\tilde{f}_*([\widetilde{M}])=f_*([M])$. Then killing $\ker \pi_1(\tilde{f})$ we obtain as before an $n$-dimensional manifold $N$ with $\pi_1(N)=G$ and $\TC(N)=2\dim(N)$. Note that we could use this procedure for all the examples above of dimension $4$ or $5$.

\newcommand{\arxiv}[1]{{\texttt{\href{http://arxiv.org/abs/#1}{{arXiv:#1}}}}}

\newcommand{\MRh}[1]{\href{http://www.ams.org/mathscinet-getitem?mr=#1}{MR#1}}

\bibliographystyle{plain}

\end{document}